\documentclass{amsart}

\usepackage{amsrefs}
\usepackage[all]{xy}
\usepackage{syntonly}
\usepackage{enumerate}
\usepackage{hyperref}
\usepackage{amsfonts}
\usepackage{amssymb}
\usepackage{amsmath}
\usepackage{amsthm}
\usepackage[inline]{enumitem}
\usepackage{tikz}
\usepackage{graphics}
\usepackage{graphicx}
\usepackage{color}
\usepackage{comment}
\usepackage{caption,lipsum}

\usepackage{lineno}

\usepackage{multirow}

\newtheorem{theorem}{Theorem}[section]
\newtheorem{corollary}[theorem]{Corollary}

\newtheorem{definition}[theorem]{Definition}
\newtheorem{lemma}[theorem]{Lemma}

\newtheorem{conjecture}[theorem]{Conjecture}

\newtheorem{observation}[theorem]{Observation}
\newtheorem{question}[theorem]{Question}

\newtheorem{fact}[theorem]{Fact}
\newtheorem{convention}[theorem]{Convention}

\specialcomment{com}{ \color{red} }{ \color{black}} 
\specialcomment{oldcom}{ \color{brown} }{\color{black}} 

\excludecomment{oldcom} 

\usepackage{stmaryrd}

\begin{document}
\title[On the universality of the nonstationary ideal]{On the universality of the nonstationary ideal}

\author{Sean D. Cox}
\email{scox9@vcu.edu}
\address{
Department of Mathematics and Applied Mathematics \\
Virginia Commonwealth University \\
1015 Floyd Avenue \\
Richmond, Virginia 23284, USA 
}

\thanks{The author gratefully acknowledges support from the VCU Presidential Research Quest Fund, and thanks Brent Cody and Monroe Eskew for helpful conversations on topics related to the paper.  The author also is grateful to the anonymous referee for suggesting Question \ref{q_Referee}.}


\begin{abstract}
Burke~\cite{MR1472122} proved that the generalized nonstationary ideal, denoted NS, is universal in the following sense: every normal ideal, and every tower of normal ideals of inaccessible height, is a canonical Rudin-Keisler projection of the restriction of $\text{NS}$ to some stationary set.  We investigate how far Burke's theorem can be pushed, by analyzing the universality properties of NS with respect to the wider class of \emph{$\mathcal{C}$-systems of filters} introduced by Audrito-Steila~\cite{AudritoSteila}.  First we answer a question of \cite{AudritoSteila}, by proving that $\mathcal{C}$-systems of filters do not capture all kinds of set-generic embeddings.  We provide a characterization of supercompactness in terms of short extenders and canonical projections of NS, without any reference to the strength of the extenders; as a corollary, NS can consistently fail to canonically project to arbitrarily strong short extenders.  We prove that $\omega$-cofinal towers of normal ultrafilters---e.g.\ the kind used to characterize I2 and I3 embeddings---are well-founded if and only if they are canonical projections of NS.  Finally, we provide a characterization of ``$\aleph_\omega$ is Jonsson" in terms of canonical projections of NS.    
\end{abstract}

\maketitle

\tableofcontents

\section{Introduction}

Burke, improving an earlier theorem of Foreman, proved the following theorem about the generalized Rudin-Keisler order on normal ideals:\begin{theorem}[Burke~\cite{MR1472122}]\label{thm_Burke}
If $\mathcal{I}$ is a normal, fine, countably complete ideal, or a tower of such ideals of inaccessible height, then there exists a stationary set $S_{\mathcal{I}}$ such that $\mathcal{I}$ is the canonical projection of the generalized nonstationary ideal restricted to $S_{\mathcal{I}}$ (denoted $\text{NS} \restriction S_{\mathcal{I}}$).
\end{theorem}
For example, if $\mathcal{I}$ is a normal ideal on $\wp_\kappa(\lambda)$, then there is some stationary $S_{\mathcal{I}} \subseteq \wp_\kappa(H_{(2^\lambda)^+})$ such that, letting $\pi$ denote the map $M \mapsto M \cap \lambda$ defined on $\wp_\kappa(H_{(2^\lambda)^+})$, we have:
\[
A \in \mathcal{I} \ \iff \  \pi^{-1}[A] \in \text{NS} \restriction S_{\mathcal{I}} 
\]
This is true even if $\mathcal{I}$ is the dual of a $\lambda$-supercompact measure on $\wp_\kappa(\lambda)$.


Claverie~\cite{claverieThesis} generalized the concept of ideals on regular cardinals to so-called \textbf{(short) ideal extenders}.  This was further generalized by Audrito-Steila~\cite{AudritoSteila} to \textbf{$\mathcal{C}$-systems of filters}, which (for technical reasons discussed in Section \ref{sec_AudritoSteila}) we will call \textbf{normal systems of filters}.  Normal systems of filters encompass single ideals, towers of ideals, and ideal extenders into a single framework.   Moreover, their framework provides a natural extension of the notion of canonical projections of ideals.  Roughly, we say that \textbf{a normal system $\boldsymbol{\vec{F}}$ of filters is a canonical projection of NS} if and only if there exists a stationary set $S_{\vec{F}}$ such that $\vec{F}^\frown \langle \text{NS} \restriction S_{\vec{F}}\rangle$ is a normal system of filters.\footnote{See Section \ref{sec_CanProjNS} for the precise definitions of \emph{normal system of filters} and  \emph{canonical projections of NS}.}  In the case where $\vec{F}$ is (the dual of) a single ideal or tower of ideals, this notion of canonical projection agrees with the usual notion used in Burke's theorem.

This paper addresses how far Burke's Theorem \ref{thm_Burke} can be pushed, namely:
\begin{question}\label{q_MainQuestion}
Which normal systems of filters are canonical projections of NS?
\end{question}

Unlike Burke's Theorem \ref{thm_Burke}, the answer to Question \ref{q_MainQuestion} depends on the normal system of filters in question, and on the ambient large cardinals in the universe.  The following list summarizes our main results regarding Question \ref{q_MainQuestion}:
\begin{enumerate}
 \item A cardinal $\kappa$ is supercompact if and only if there are class many $\lambda$ such that there exists a short $(\kappa,\lambda)$-extender that is a canonical projection of NS.  In particular, if $\kappa$ is strong but not supercompact, then there are arbitrarily strong short extenders with critical point $\kappa$ that are \textbf{not} canonical projections of NS.  One interesting feature of this characterization of supercompactness is that it makes no reference to the strength of the extenders.  See Section \ref{sec_CharSuper}.
 \item Every precipitous $\omega$-cofinal tower of normal ideals is a canonical projection of NS.  Moreover, for $\omega$-cofinal towers of normal \textbf{ultra}filters, being a canonical projection of NS exactly characterizes having a wellfounded ultrapower.  In particular, I2 embeddings can be characterized by so-called I3 towers that happen to be canonical projections of NS.  See Section \ref{sec_OmegaCofTowers}.
 \item If $2^\omega < \aleph_\omega$, then $\aleph_\omega$ is Jonsson if and only if there exists an I2 tower of normal ideals that is a canonical projection of NS and decides its critical sequence.  This extends an old theorem of Silver.  See Section \ref{sec_OmegaCofTowers}.
\end{enumerate}

We also answer a question of Audrito-Steila~\cite{AudritoSteila} by proving that normal systems of filters do not capture all kinds of set generic embeddings.  This appears in Section \ref{sec_AudritoSteila}.

All terminology and notation agrees with Jech~\cite{MR1940513}, unless otherwise indicated.  By \textbf{ideal} we will always mean a fine, countably complete ideal, but not necessarily normal.  The \textbf{support} of an ideal $I$ is the set $\cup S$, where $S$ is any $I$-positive set.\footnote{Since $I$ is fine, this does not depend on the choice of $S$.  E.g.\ if $I$ is an ideal on $\wp_\kappa(H_\theta)$ then the support of $I$ is $H_\theta$, and $H_\theta = \cup S$ where $S$ is any $I$-positive set.}  If $I$ is an ideal on a set $Z$ then $\mathbb{B}_I$ denotes the quotient boolean algebra $(\wp(Z)/I, \le_I)$.  A set $S$ is \textbf{stationary} iff for every $F: [\cup S]^{<\omega} \to \cup S$ there is some $M \in S$ that is closed under $F$.  If $S$ is a stationary set, the \textbf{nonstationary ideal restricted to $\boldsymbol{S}$, denoted $\boldsymbol{\text{NS} \restriction S}$}, is the ideal generated by
\[
\big\{ \wp(\cup S) \setminus \mathcal{C}_{F} \ : \   F: [\cup S]^{<\omega} \to \cup S \big\}
\] 
where
\[
\mathcal{C}_F:= \{ X \subseteq \cup S \ : \ X \text{ is closed under } F  \}
\]

\section{Canonical projections of NS}\label{sec_CanProjNS}

The notion of one normal ideal canonically projecting to another, or of a normal ideal projecting to a tower of normal ideals, has appeared many times in the literature, e.g.\ Foreman~\cite{MattHandbook} and Burke~\cite{MR1472122}.  For concreteness, suppose $J$ is a normal ideal with support $\lambda_J$, $I$ is a normal ideal with support $\lambda_I$, and $\lambda_J \le \lambda_I$.  We say that $J$ is the canonical projection of $I$ to $\lambda_J$ iff, letting $\pi: \wp(\lambda_I) \to \wp(\lambda_J)$ denote the map $M \mapsto M \cap \lambda_J$, we have
\[
J = \big\{   A \in \wp\wp(\lambda_J) \ : \ \pi^{-1}[A] \in I  \big\}
\]
It will be illustrative to give an alternative characterization: $J$ is the canonical projection of $I$ iff for all $A \in \wp \wp(\lambda_J)$:
\begin{equation}\label{eq_DerivedIdeal}
A \in J \iff \Vdash_{\mathbb{B}_I} j_{\dot{G}}[\lambda_J] \notin j_{\dot{G}}(\check{A})
\end{equation}
where $j_{\dot{G}}$ is the $\mathbb{B}_I$-name for its generic ultrapower embedding.  Similarly, $I$ projects to a tower $\vec{J} = \langle J_\alpha \ : \ \alpha < \delta \rangle$ iff $I$ projects to every $J_\alpha$.  But what should it mean for a normal ideal $I$ to project to an ideal \emph{extender} $\langle J_a \ : \ a \in [\lambda_I]^{<\omega} \rangle$, which looks very different from a tower of normal ideals?\footnote{In fact many $J_a$ in the ideal extender are typically \textbf{not} normal ideals, even if the entire ideal extender satisfies the so-called normality criterion for extenders.}   Roughly, we will define this to mean that $\vec{J}$ is derived from $I$'s generic ultrapower embeddings in the following way:  that for all $a \in [\lambda_I]^{<\omega}$,
\begin{equation}\label{eq_DerivedExtender}
X \in J_a \ \iff \ \Vdash_{\mathbb{B}_I} \ \ \check{a} \notin j_{\dot{G}}(\check{X}) 
\end{equation}
where $j_{\dot{G}}$ is the $\mathbb{B}_I$-name for the generic ultrapower embedding.    Although \eqref{eq_DerivedExtender} seems vaguely similar to \eqref{eq_DerivedIdeal}, to see that it really is the same phenomenon, we will use a notion of Audrito-Steila~\cite{AudritoSteila}, which they call $\boldsymbol{\mathcal{C}}$\textbf{-systems of filters}, which elegantly unifies normal ideals, towers of normal ideals, and ideal extenders into a single framework.

A set $\mathcal{C}$ is called a \textbf{directed set of domains} iff $\mathcal{C}$ is closed under subsets and (finite) unions, and $\bigcup \mathcal{C}$ is transitive.  For each $a \in \mathcal{C}$, 
\[
O_a:= \{ \pi_M \restriction (a \cap M) \ : \ M \subset \text{trcl}(a) \ \wedge \ (M,\in) \text{ is extensional } \}
\]
where $\pi_M$ denotes the Mostowski collapsing map of $M$. If $a \subseteq b$ then $\pi_{b,a}: O_b \to O_a$ is defined by $f \mapsto f \restriction a$.  A \textbf{directed system of filters (with domain $\boldsymbol{\mathcal{C}}$)} is a sequence $\vec{F} = \langle F_a \ : \ a \in \mathcal{C} \rangle$ such that $\mathcal{C}$ is a directed set of domains, each $F_a \subset \wp(O_a)$ is a filter, and whenever $a \subseteq b$ are both in $\mathcal{C}$ then $F_a$ is the canonical projection of $F_b$ via $\pi_{b,a}$; i.e.\ for all $A \subseteq O_a$:
 \[
 A \in F_a \ \iff \ \pi_{b,a}^{-1}[A] \in F_b
 \]

If $\vec{F}$ is a directed system of filters with domain $\mathcal{C}$, the \textbf{poset associated with $\boldsymbol{\vec{F}}$, denoted $\boldsymbol{\mathbb{B}_{\vec{F}}}$}, is the partial order whose conditions are pairs $(a,S)$ where $a \in \mathcal{C}$ and $S \in F_a^+$ (i.e.\ $S \subseteq O_a$ and $S \notin \text{Dual}(F_a)$).  The ordering of $\mathbb{B}_{\vec{F}}$ is:  $(a,S) \le (b,T)$ iff 
\[
\pi_{a \cup b, a}^{-1}[S] \le_{F_{a \cup b}} \pi_{a \cup b, b}^{-1}[T]
\]
where $X \le_F Y$ means $X \setminus Y$ is in the dual ideal of $F$.  If $G$ is $(V,\mathbb{B}_{\vec{F}})$-generic, then for each $a \in \mathcal{C}$ the object
\[
G_a:= \{  S \subset O_a \ : \ (a,S) \in G    \}
\]
is an ultrafilter on $\wp^V(O_a)$ extending $F_a$.  Hence it makes sense to form the ultrapower map $j_{G_a}:V \to \text{ult}(V,G_a)$, and if $a \subseteq b$ are both in $\mathcal{C}$ then $\pi_{b,a}$ gives rise to an elementary embedding $k_{a,b}: \text{ult}(V,G_a) \to \text{ult}(V,G_b)$ such that $j_{G_b} = k_{a,b} \circ j_{G_a}$.  Since $\mathcal{C}$ is $\subseteq$-directed, in $V[G]$ the ultrapowers of $V$ by the various $G_a$'s gives rise to a directed system and an ultrapower map $j_G: V \to \text{ult}(V,G)$ where the target $\text{ult}(V,G)$ is the direct limit of the $\text{ult}(V,G_a)$'s.  The map $j_G$ is called the \textbf{generic ultrapower map for $\boldsymbol{\vec{F}}$}.

Definition 2.6 of Audrito-Steila~\cite{AudritoSteila} imposes an additional ``normality" requirement of directed systems of filters, resulting in what they call a \textbf{$\boldsymbol{\mathcal{C}}$-system of filters}, which is the main focus of their paper.  Rather than describing their somewhat complicated normality criterion,\footnote{i.e.\ requirement 4 of Definition 2.6 of \cite{AudritoSteila}.} we choose instead to give an equivalent definition that will be more convenient for our applications.  Also, because we will often be viewing one such system as a projection of another---e.g.\ viewing some system $\vec{F} = \langle F_a \ : \ a \in \mathcal{C}_{\vec{F}} \rangle$ as a projection of another system $\vec{H} = \langle H_a \ : \ a \in \mathcal{C}_{\vec{H}} \rangle$, it would lead to confusion if we referred to both of them as ``$\mathcal{C}$-systems of filters", since $\mathcal{C}_{\vec{F}}$ and $\mathcal{C}_{\vec{H}}$ will differ.  Therefore, what Audrito-Steila call a ``$\mathcal{C}$-system of filters", we will instead call a ``normal system of filters":

\begin{definition}\label{def_NormalSystem}
A \textbf{normal system of filters} is a sequence $\vec{F} = \langle F_a \ : \ a \in \mathcal{C} \rangle$ such that: 
\begin{enumerate}
 \item $\mathcal{C}$ is a directed set of domains and $\vec{F}$ is a directed system of filters with domain $\mathcal{C}$ (as defined above);
 \item $\mathbb{B}_{\vec{F}}$ forces $\bigcup \mathcal{C}$ to be a subset of the wellfounded part of $\text{ult}(V,\dot{G})$;\footnote{Where we require that the wellfounded part of $\text{ult}(V,\dot{G})$ has been transitivized.}
 \item\label{item_DefOfFilter} For all $a \in \mathcal{C}$ and all $D \subset O_a$:  
 \[
 D \in F_a \ \iff \ \Vdash_{\mathbb{B}_{\vec{F}}}  \big( j_{\dot{G}} \restriction \check{a} \big)^{-1} \in j_{\dot{G}}(\check{D})
 \]
\end{enumerate}
\end{definition}

The following lemma follows directly from Proposition 2.24 of \cite{AudritoSteila}:
\begin{lemma}
A sequence $\vec{F}$ is a normal system of filters as in Definition \ref{def_NormalSystem} if and only if $\vec{F}$ is a $\mathcal{C}$-system of filters as in Definition 2.1 of Audrito-Steila~\cite{AudritoSteila}. 
\end{lemma}

Normal systems of filters unify several commonly-used notions that typically have very different representations.  The following is our ``official" definition of ideal extenders, normal towers, and normal ideals, though as noted below we will often switch to other, more traditional, representations:
\begin{definition}[Audrito-Steila~\cite{AudritoSteila}]\label{def_AS_Various}
Let $\vec{F} = \langle F_a \ : \ a \in \mathcal{C} \rangle$ be a normal system of filters with domain $\mathcal{C}$.  Then $\vec{F}$ is called:
\begin{itemize}
 \item an \textbf{ideal extender} if  $\mathcal{C} = [\lambda]^{<\omega}$ for some ordinal $\lambda$; an \textbf{extender} is an ideal extender where every $F_a$ is an ultrafilter;
 \item a \textbf{normal tower} if $\mathcal{C} = V_\lambda$ for some $\lambda$;
 \item a \textbf{normal ideal} if $\mathcal{C}$ has a $\subseteq$-largest element; equivalently,\footnote{Because directed sets of domains by definition are required to be closed under subsets.} $\mathcal{C} =  \wp(X) $ for some transitive set $X$.\footnote{The transitivity requirement is not a real restriction, because any ideal is isomorphic to an ideal with transitive support.}  
\end{itemize}
\end{definition}

As pointed out in Section 2.1 of \cite{AudritoSteila}, these notions agree with the standard definitions via the following correspondences.   A single normal ideal $I$ with (transitive) support $X$ in the usual sense corresponds to a normal system of filters $\langle F^I_a \ : \ a \in \wp(X) \rangle$ where 
\[
F^I_X:= \big\{ \{ \pi_M \ : \ M \in D \} \ : \ D \in \text{Dual}(I)   \big\}
\]
and for each $a \in \wp(X)$, $F^I_a$ is the dual of the canonical projection of $I$ to $O_a$ (via the function $\pi_{X,a}$).  Here $\pi_M$ denotes the Mostowski collapsing map of $M$.  Of course, the entire system here is generated by the single filter $F_X$, which is essentially the dual of the original ideal $I$.

If $\vec{F}$ is an ideal extender in the sense of Definition \ref{def_AS_Various}, then 
\[
\langle E_a \ : \ a \in [\lambda]^{<\omega} \rangle
\]
is an ideal extender in the sense of Claverie~\cite{claverieThesis}, where
\[
E_a:= \Big\{  \{ \text{rng}(f) \ : \ f \in D \text{ and } \text{rng}(f) \in [\kappa_a]^{|a|} \}   \ : \ D \in F_a        \Big\}
\]
and $\kappa_a$ is the cardinality of the smallest $F_a$-measure one set.  We will call $\vec{E}$ the \textbf{Claverie-style representation of $\boldsymbol{\vec{F}}$}.  Moreover, the Claverie-style forcing associated with $\vec{E}$ is equivalent to $\mathbb{B}_{\vec{F}}$ and the generic ultrapower maps are identical.  Conversely, if $\langle E_a \ : \ a \in [\lambda]^{<\omega} \rangle$ is an ideal extender in the sense of Claverie~\cite{claverieThesis}, then one can use the Claverie-style forcing $\mathbb{B}_{\vec{E}}$ associated with $\vec{E}$---which forces a generic embedding with $\lambda$ contained in the wellfounded part of the generic ultrapower---to define a sequence $\vec{F} = \langle F_a \ : \ a \in [\lambda]^{<\omega} \rangle$ by:  $X \in F_a$ iff $X \subset O_a$ and
\[
\Vdash_{\mathbb{B}_{\vec{E}}}  \ (j_{\dot{H}} \restriction \check{a})^{-1} \in j_{\dot{H}}(\check{X}) 
\]
where $j_{\dot{H}}$ is the $\mathbb{B}_{\vec{E}}$-name for the generic ultrapower map.  Then $\mathbb{B}_{\vec{F}}$ is forcing equivalent to $\mathbb{B}_{\vec{E}}$ and their generic ultrapowers are the same.

Similarly, normal towers in the sense of Definition \ref{def_AS_Various}, and normal towers in the classic sense (e.g.\ as in \cite{MattHandbook}), can be put into correspondence as follows:  if $\vec{F} = \langle F_a \ : \ a \in V_\lambda \rangle$ is a tower as in Definition \ref{def_AS_Various}, then 
\[
\langle E_a \ : \ a \in V_\lambda \rangle
\]
is a normal tower in the classic sense, where 
\[
E_a := \big\{  \{ \text{dom}(f) \ : \ f \in D  \} \ : \ D \in F_a   \big\}
\]
Conversely, if $\vec{E} = \langle E_a \ : \ a \in V_\lambda \rangle$ is a normal tower as in \cite{MattHandbook}, then
\[
\vec{F} = \langle F_a \ : \ a \in V_\lambda \rangle
\] 
is a normal tower in the sense of Definition \ref{def_AS_Various}, where
\[
F_a := \big\{  \{ \pi_M \ : \ M \in D    \} \ : \ D \in F_a   \big\}
\]
and $\pi_M$ denotes the transitive collapsing map of $M$.  Alternatively, one could instead use the (classic) generic ultrapower by the tower $\vec{E}$ in order to define $\vec{F}$, similarly to how it was done for ideal extenders.

\begin{definition}\label{def_PO}
Given normal systems of filters $\vec{F}$ and $\vec{H}$, we write $\vec{F} \ge \vec{H}$---and say \textbf{$\boldsymbol{\vec{F}}$ canonically projects to $\boldsymbol{\vec{H}}$}---iff $\vec{F}$ extends $\vec{H}$, i.e.\ $\vec{F} \restriction \text{dom}(\vec{H}) = \vec{H}$.

We write
\[
\vec{F} \sim \vec{H}
\]
and say that \textbf{$\boldsymbol{\vec{F}}$ is equivalent to $\boldsymbol{\vec{H}}$ (as normal systems of filters)}, iff there exists an isomorphism
\[
\pi: \text{ro}\big( \mathbb{B}_{\vec{F}} \big) \to \text{ro} \big( \mathbb{B}_{\vec{H}} \big)
\]
such that whenever $G$ is generic for $\mathbb{B}_{\vec{F}}$, then $j_{G} = j_{\pi[G]}$ and $\text{ult}(V,G) = \text{ult}(V,\pi[G])$.  This is clearly an equivalence relation.
\end{definition}

In the case where $\vec{F}$ and $\vec{H}$ are both normal ideals (or normal towers), the notion of canonical projection from Definition \ref{def_PO} agrees with the usual notion of canonical ideal projection (modulo the different representation of normal ideals and normal towers described above).  The main new content of Definition \ref{def_PO} is when the system being projected to---i.e.\ the $\vec{H}$ in the definition---is \emph{neither} a normal ideal \emph{nor} a tower of normal ideals.  For example, we will often be interested in cases where the $\vec{F}$ from Definition \ref{def_PO} is a normal ideal, but the $\vec{H}$ from Definition \ref{def_PO} is an ideal extender.

We remark that it is possible to have $\vec{F} \sim \vec{H}$ while $\vec{H} > \vec{F}$.  For example, suppose $\vec{F}$ is just a normal measure on $\kappa$.  One can then derive an extender of any length desired from $j_{\vec{F}}$, and this extender will be equivalent to $\vec{F}$, yet its domain will properly extend the domain of $\vec{F}$.

%


Because of the correspondences described immediately after Definition \ref{def_AS_Various}, for convenience we will often switch between different representations, as in the following convention:

\begin{convention}\label{conv_IdealBelow}
Given a normal ideal $I$ with transitive support $X$, and a normal system of filters $\vec{F}$, when we write
\[
I \ge \vec{F}
\]
it is understood that we are representing $I$ as the normal system of filters $\langle E^I_a \ : \ a \in \wp(X) \rangle$ as described above.  So by $I \ge \vec{F}$ we really mean
\[
\langle E^I_a \ : \ a \in \wp(X) \rangle \ge \vec{F}
\]
\end{convention}

\begin{definition}\label{def_CanProjOfNS}
Let $\vec{F}$ be a normal system of filters.  We say that \textbf{$\boldsymbol{\vec{F}}$ is a canonical projection of NS} iff there exists some stationary set $S$ such that 
\[
\text{NS} \restriction S \ \ge \ \vec{F}
\]
as in Convention \ref{conv_IdealBelow}.
\end{definition}


\begin{definition}
Given any class $\Gamma$ of normal systems of filters, we say that \textbf{NS is universal for $\boldsymbol{\Gamma}$} iff every member of $\Gamma$ is a canonical projection of NS.
\end{definition}

Then Burke's Theorem \ref{thm_Burke} can be rephrased as:  \emph{NS is universal for the class of all normal ideals and normal towers of inaccessible height.}

As with Burke's argument for Theorem \ref{thm_Burke}, it will be convenient to be able to characterize Definition \ref{def_CanProjOfNS} in terms of ``good" structures (as introduced by Foreman-Magidor~\cite{MR1359154}).  We also generalize self-genericity to such systems:

\begin{definition}\label{def_GoodSelfGen}
Given a normal system $\vec{F} = \langle F_a \ : \ a \in \mathcal{C} \rangle$ of filters, and some $M \prec (H_\theta, \in, \vec{F})$ where $\vec{F} \in H_\theta$, let $\sigma_M: H_M \to H_\theta$ be the inverse of the  Mostowski collapse map of $M$.  For each $b \in M$ we let $b_M$ denote $\sigma_M^{-1}(b)$.  For each $a \in M \cap \mathcal{C}$ let
\[
 U^{M,a}:= \{ X \in \wp^{H_M}\big( (O_a)_M \big) \ : \  \sigma_M^{-1} \restriction a \in \sigma_M(X)   \}
\]
and let $\vec{U}^M:= \langle U^{M,a} \ : \ a \in M \cap \mathcal{C} \rangle$.

We say that:
\begin{itemize}
 \item \textbf{$\boldsymbol{M}$ is $\boldsymbol{\vec{F}}$-good} iff $\vec{U}^M$ pointwise extends the system $\big( \vec{F} \big)_M$; i.e.\ iff $U^{M,a} \supseteq (F_a)_M$ for every $a \in M \cap \mathcal{C}$.
 \item \textbf{$\boldsymbol{M}$ is $\boldsymbol{\vec{F}}$-self-generic} iff $\vec{U}^M$ is generic over $H_M$ for the poset $\big( \mathbb{B}_{\vec{F}} \big)_M$, where $\mathbb{B}_{\vec{F}}$ is the forcing associated with $\vec{F}$ as described before Definition \ref{def_NormalSystem}.
\end{itemize}
\end{definition}
It is routine to check that:
\begin{itemize}
 \item $M$ is $\vec{F}$-good iff for all $a \in M \cap \mathcal{C}$ and all $D \in M \cap F_a$: $\pi_M \restriction a \in D$.
 \item $M$ is $\vec{F}$-self-generic iff for every $\mathcal{A} \in M$ that is a maximal antichain in $\mathbb{B}_{\vec{F}}$, there is some $(a,S) \in \mathcal{A} \cap M$ such that $\pi_M \restriction a \in S$.  
\end{itemize}

The following technical fact will be used in Section \ref{sec_CharSuper}:
\begin{fact}\label{fact_TransBelowKappa}
Suppose $\vec{F}$ is an ideal extender with critical point $\kappa$, and $M \prec (H_\theta,\in,\vec{F})$ is $\vec{F}$-good.  Then $M \cap \kappa \in \kappa$.
\end{fact}
\begin{proof}
Suppose not; then letting $\eta_M$ be the least element of $M$ that is not a subset of $M$, we have $\eta_M < \kappa$ and
\begin{equation}\label{eq_SendsEtaM}
\pi_M(\eta_M) < \eta_M \ 
\end{equation}
where $\pi_M$ is the Mostowski collapsing map of $M$.  

Let $C:= \big\{ \{ \langle \eta_M, \eta_M \rangle \} \big\}$; i.e.\ $C$ is the singleton whose only element is the function with domain $\{ \eta_M \}$ that sends $\eta_M$ to itself.  Note that $C \in M$.  Since $\kappa$ is the critical point of $\vec{F}$ and $\eta_M < \kappa$, the generic ultrapower always fixes $\eta_M$, so
\[
\Vdash_{\mathbb{B}_{\vec{F}}} \ (j_{\dot{G}} \restriction \{ \eta_M \})^{-1}  = \{ \langle \eta_M, \eta_M \rangle \} \in j_{\dot{G}}(C)  = C
\]
This implies that $C \in F_{ \{ \eta_M \} }$.  Since both $C$ and $\{ \eta_M \}$ are in $M$ and $M$ is $\vec{F}$-good, then $\pi_M \restriction \{ \eta_M \} \in C$; i.e.\ $\pi_M$ fixes $\eta_M$.  This contradicts \eqref{eq_SendsEtaM}.
\end{proof}

If $\vec{U}$ is a normal system of \emph{ultra}filters and $M \prec (H_\theta, \in, \vec{U})$, then  $\mathbb{B}_{\vec{U}}$ and thus its preimage under $\sigma_M$ is the trivial forcing.  It follows that:
\begin{observation}
Suppose $\vec{U}$ is a normal system of ultrafilters.  Then for any $M \prec (H_\theta, \in, \vec{U})$: $M$ is $\vec{U}$-good if and only if $M$ is $\vec{U}$-self-generic.
\end{observation}

The next observation will be used in Section \ref{sec_CharSuper}:
\begin{observation}\label{obs_DerivedIsElement}
Suppose $\vec{U}$ is a normal system of ultrafilters and that $M \prec (H_\theta, \in, \vec{U})$ is $\vec{U}$-good.  Then $\vec{U}^M$---i.e.\ the normal system derived from $\sigma_M$ as in Definition \ref{def_GoodSelfGen}---is an element of $H_M$.  
\end{observation}
\begin{proof}
Let $\vec{U}_M:= \sigma_M^{-1}(\vec{U})$ where $\sigma_M: H_M \to H_\theta$ is the inverse of the Mostowski collapse of $M$.  We prove that $\vec{U}_M = \vec{U}^M$.  Pick any $a \in M \cap \mathcal{C}$ and let $a_M:= \sigma_M^{-1}(a)$.  Now both $U^M_{a_M}$ and $(U_M)_{a_M}$ are ultrafilters on $\wp^{H_M}(O_{a_M})$, so to prove they are equal it suffices to show that one of them extends the other.  And $U^M_{a_M}$ extends $(U_M)_{a_M}$ because $M$ is $\vec{U}$-good.
\end{proof}

Next we generalize Foreman's~\cite{MattHandbook} construction of the ``conditional club filter relative to a normal ideal", where we replace the normal ideal by any normal system of filters.  The main new content of this definition is when the normal system of filters is an ideal extender.  The following notation is the appropriate analogue of the notation of Feng-Jech~\cite{MR1668171} in the current setting:
\begin{definition}\label{def_over}
Given a collection $R$ of elementary submodels of some $H_\theta$, and given sets $a$ and $T$, set
\[
R \searrow_a T:= \{ M \in R \ : \ \pi_M \restriction a \in T \}
\]
where $\pi_M$ denotes the transitive collapsing map of $M$.
\end{definition}

\begin{definition}
Given a normal system $\vec{F} = \langle F_a \ : \ a \in \mathcal{C} \rangle$ of filters, and a collection $R$ of elementary submodels of some $H_\theta$, we say that \textbf{$\boldsymbol{R}$ is $\boldsymbol{\vec{F}}$-projective stationary} iff for every $a \in \mathcal{C}$ and every $T \in F_a^+$, the set $R \searrow_a T$ (as defined in Definition \ref{def_over}) is stationary.
\end{definition}

\begin{lemma}\label{lem_Equiv_ProjNS}
Let $\vec{F}=\langle F_a \ : \ a \in \mathcal{C} \rangle$ be a normal system of filters.  The following are equivalent:
\begin{enumerate}
 \item\label{item_SomeIdeal} There exists some normal ideal $I$ such that $I \ge \vec{F}$ as in Convention \ref{conv_IdealBelow};
 \item\label{item_ProjOfNS} $\vec{F}$ is a canonical projection of NS as in Definition \ref{def_CanProjOfNS};
 \item\label{item_ProjStat} For all sufficiently large $\theta$, the set of $\vec{F}$-good substructures of $H_\theta$ is $\vec{F}$-projective-stationary.
\end{enumerate}

\textbf{Moreover:}  if each $F_a$ is an ultrafilter, then replacing ``$\vec{F}$-projective-stationary" in item \ref{item_ProjStat} with ``stationary" results in an equivalent statement.
\end{lemma}
\begin{proof}
\ref{item_ProjOfNS} implies \ref{item_SomeIdeal} because the nonstationary ideal is normal.  To see that  \ref{item_SomeIdeal} implies  \ref{item_ProjOfNS}: given a normal $I$ such that $I \ge \vec{F}$, let $S_I$ be the stationary set from Burke's theorem such that $\text{NS} \restriction S_I \ge I$.  The relation in Definition \ref{def_PO} is clearly transitive, so $\text{NS} \restriction S_I \ge \vec{F}$.

Next we prove \ref{item_ProjOfNS} $\implies$ \ref{item_ProjStat}.  Let $S$ be the stationary set witnessing \ref{item_ProjOfNS}, i.e.\ so that $\text{NS} \restriction S$ projects canonically to $\vec{F}$.  By lifting stationary sets if necessary, we can WLOG assume that $S$ is a stationary collection of elementary substructures of some $H_\theta$.  Identifying $\text{NS} \restriction S$ as a normal system of filters with largest member
\[
F_{H_\theta} = \Big\{ \{ \pi_M \ : \ M \in C  \} \ : \ C \in \text{Dual}(\text{NS} \restriction S)   \Big\}
\]
as described after Definition \ref{def_AS_Various}, we can translate assumption \ref{item_ProjOfNS} as follows:
\begin{align}\label{eq_Translate}
\begin{split}
\forall a \in \mathcal{C} \ \forall D \subseteq O_a \  \   D \in F_a & \iff \pi_{H_\theta,a}^{-1}[D] \in F_{H_\theta} \\
& \iff \{ \pi_M \ : \ M \in S \text{ and }  \pi_M \restriction a \in D   \} \in F_{H_\theta} \\
& \iff \{ M \in S \ : \ \pi_M \restriction a \in D   \} \in \text{Dual}(\text{NS} \restriction S) \\
& \iff S \searrow_a D \in \text{Dual}(\text{NS} \restriction S)
\end{split}
\end{align}

Let $R$ denote the collection of $\vec{F}$-good members of $S$, let $a \in \mathcal{C}$, and $T \in F^+_a$; we need to prove that $R \searrow_a T$ is stationary.  Now $S \searrow_a T$ is a stationary subset of $S$ by \eqref{eq_Translate}, so to finish the proof it suffices to prove that $S \setminus R$ is nonstationary.  Suppose toward a contradiction that $S \setminus R$ is stationary; then for stationarily many $M \in S$ there is a pair $(a_M,D_M) \in M \cap (\mathcal{C} \times F_{a_M})$ such that $\pi_M \restriction  a_M \notin D_M$.  Then by Fodor's Lemma there is a stationary $S' \subseteq S$ and some fixed $(a^*, D^*) \in \mathcal{C} \times F_{a^*}$ such that 
\begin{equation}\label{eq_NotInDstar}
\pi_M \restriction a^* \notin D^* \text{ for all } M \in S'
\end{equation}  

Since $D^* \in F_{a^*}$, by \eqref{eq_Translate} we have
\[
\widetilde{D}^*:=  \{ M \in S \ : \ \pi_M \restriction a^* \in D^*  \} \in \text{Dual}(\text{NS} \restriction S) .
\]
But \eqref{eq_NotInDstar} implies $\widetilde{D}^* \cap S' =\emptyset$, contradicting that $S'$ is a stationary subset of $S$.
 
Finally we prove \ref{item_ProjStat}
 $\implies$ \ref{item_ProjOfNS}.  Let $S$ denote the set of $\vec{F}$-good elementary substructures of $H_\theta$, and again identify $\text{NS} \restriction S$ with a normal system with largest member $F_{H_\theta}$ .  We need to prove that
 
 \[
 \forall a \in \mathcal{C} \ \ \forall D \subseteq O_a \ D \in F_a \iff \pi^{-1}_{H_\theta,a}[D] \in F_{H_\theta} 
 \] 
 which (by the same argument as \eqref{eq_Translate}) is equivalent to proving 
 \[
 \forall a \in \mathcal{C} \ \ \forall D \subseteq O_a \ D \in F_a \iff S \searrow_a D \in \text{Dual}(\text{NS} \restriction S) .
 \] 
So fix some $a \in \mathcal{C}$ and some $D \subseteq O_a$.  If $D \in F_a$ then for every $M \in S$ with $(a,D) \in M$, we have $\pi_M \restriction a \in D$ by definition of goodness.  This collection constitutes a measure one set in $\text{NS} \restriction S$.  If $D \notin F_a$ then $D^c \in F_a^+$, so assumption \eqref{item_ProjStat} ensures there are stationarily many $M \in S$ such that $\pi_M \restriction a \in D^c$.  So  
\[
S \searrow_a D^c= \{ M \in S \ : \  \pi_M \restriction a \in D^c \} \text{ is stationary}
\] 
and hence the complement of $S \searrow_a D^c$---i.e.\ $S \searrow_a D$---is not in the dual filter of $\text{NS} \restriction S$.

For the ``moreover" part:  let $R$ denote the set of $\vec{F}$-good elementary substructures of $H_\theta$, and assume that $R$ is stationary; we need to prove that $R$ is $\vec{F}$-projective stationary.  So let $a \in \mathcal{C}$ and $T \in F^+_a$.  Since $F_a$ is an ultrafilter then in fact $T \in F_a$.  There are stationarily many $M \in R$ such that $(a,T) \in M$, and since any such $M$ is $\vec{F}$-good and $T \in F_a$ then $\pi_M \restriction a \in T$.
\end{proof}


It will be convenient to see how Definition \ref{def_PO} translates into the setting where we view ideal extenders in the Claverie~\cite{claverieThesis} form.  

\begin{definition}\label{def_GoodClaverie}
Let $\vec{F} = \langle F_a \ : \ a \in [\lambda]^{<\omega} \rangle$ be an ideal extender as in Definition \ref{def_AS_Various} with critical point $\kappa$, and $\vec{E} = \langle E_a \ : \ a \in [\lambda]^{<\omega} \rangle$ be its Claverie-style representation (as described after Definition \ref{def_AS_Various}).\footnote{So in particular each $E_a$ is a filter on $[\kappa_a]^{|a|}$ where $\kappa_a$ is the minimum cardinality of a set in $E_a$.}  Let $M \prec (H_\theta,\in,\vec{E})$ and $\sigma_M:H_M \to H_\theta$ be the inverse of the Mostowski collapse of $M$.  We say  \textbf{$\boldsymbol{M}$ is $\boldsymbol{\vec{E}}$-good} iff the classic $(\kappa_M, \lambda_M)$-extender $\vec{U}^M$ derived from $\sigma_M$ extends the system $\sigma_M^{-1}(\vec{E})$.  In other words, $M$ is $\vec{E}$-good iff 
\[
\big\{ X \in H_M \ : \  a \in \sigma_M(X) \big\} \supseteq \sigma_M^{-1}(E_a)
\]
for every $a \in M \cap [\lambda]^{<\omega}$.
\end{definition}

The following lemmas are routine to check, and we leave the proofs to the interested reader.
\begin{lemma}
Let $\vec{F}$ be an ideal extender as in Definition \ref{def_AS_Various}, and $\vec{E}$ be its equivalent Claverie-style representation.  Given any $M \prec (H_\theta, \in, \vec{F})$:  $M$ is $\vec{F}$-good in the sense of Definition \ref{def_GoodSelfGen} if and only if $M$ is $\vec{E}$-good in the sense of Definition \ref{def_GoodClaverie}.
\end{lemma}

\begin{lemma}\label{lem_Char_I_below}
Let $I$ be a normal ideal, $\vec{F}$ an ideal extender as in Definition \ref{def_AS_Various}, and  $\vec{E} = \langle E_a \ : \ a \in [\lambda]^{<\omega} \rangle$ the Claverie-style representation of $\vec{F}$.  The following are equivalent:
\begin{enumerate}
 \item\label{item_IbelowE} $I \ge \vec{F}$ (following Convention \ref{conv_IdealBelow});
 \item\label{item_CharProj} For all $a \in [\lambda]^{<\omega}$:
 \[
 E_a = \big\{ X \subseteq \kappa_a \ : \ \Vdash_{\mathbb{B}_I} \ \check{a} \in j_{\dot{G}}(\check{X})  \big\}
 \]
\end{enumerate}
\end{lemma}

\section{A question of Audrito-Steila}\label{sec_AudritoSteila}

Audrito-Steila~\cite{AudritoSteila} call attention to the following concepts (see Definition 3.2 of \cite{AudritoSteila}).  Given a first-order property $P(j)$ in the language $\{ \in, j \}$ (where $j$ is a predicate symbol), we say that \textbf{$\boldsymbol{\kappa}$ has generically property P} iff, in some generic extension $V[G]$ of $V$, there is a definable elementary embedding $j$ with domain $V$ and critical point $\kappa$ such that $(V[G], \in, j) \models P(j)$.  They say that \textbf{$\boldsymbol{\kappa}$ has ideally property P} iff there exists some normal system of filters $\vec{F}$ with critical point $\kappa$ such that
\[
\Vdash_{\mathbb{B}_{\vec{F}}} \ \text{``The generic ultrapower embedding has property P"}
\]
They asked:
\begin{question}[Question 4.4 of \cite{AudritoSteila}]\label{q_no}
Is having ideally property P equivalent to having generically property
P?
\end{question}
Question \ref{q_no} was apparently motivated by Claverie's~\cite{claverieThesis} distinction between ``ideally strong" and ``generically strong" cardinals.  We prove that the answer to Question \ref{q_no} is \emph{no}.

Given a (possibly external) elementary embedding $j: V \to N$, we say that \textbf{$\boldsymbol{j}$ is eventually continuous} iff there exists some $\theta$ such that for all $\alpha$ with $\text{cf}^V(\alpha) > \theta$, $j$ is continuous at $\alpha$.\footnote{i.e.\ $\text{sup}(j[\alpha]) = j(\alpha)$.}  

\begin{theorem}\label{thm_no}
Let $P(j)$ stand for:  ``$j$ is a nontrivial elementary embedding and is \textbf{not} eventually continuous."  It is consistent (relative to a proper class of measurable cardinals) that P generically holds, but does not ideally hold.
\end{theorem}
\begin{proof}
Let $V$ be a model of ZFC and $\langle U_i,\kappa_i \ : \ i \in \text{ORD} \rangle$ be a definable class such that $\langle \kappa_i \ : \ i \in \text{ORD} \rangle$ is an increasing sequence of measurable cardinals, and $U_i$ is a normal measure on $\kappa_i$ for each $i \in \text{ORD}$.  For example, this can be arranged if $V$ has a definable wellorder of the universe and has class many measurable cardinals (e.g.\ it holds for the core model below $0$-sword from \cite{ZemanBook}; in fact that model has exactly one normal measure on each measurable cardinal).  Then there is an elementary embedding $j: V \to N$ with $N$ wellfounded that is definable in $V$, and is not eventually continuous.  There are many ways to construct such a $j$, but for example, define an iteration $\langle N_i, \pi_{i,k} \ : \ i \le k \le  \text{ORD} \rangle$ as follows:
\begin{itemize}
 \item $N_0 := V$ and $\pi_{0,0} = \text{id}$
 \item If $N_i$ and $\pi_{k,i}$ are defined for all $k \le i$ then $N_{i+1}:= \text{ult}\big( N_i, \pi_{0,i}(U_i) \big)$
 \item If $i$ is a limit ordinal, or if $i = \text{ORD}$, then $N_i$ and $\pi_{k,i}$ (for $k < i$) are defined to be the direct limit.
\end{itemize}
Then $\pi_{0,\text{ORD}}: V \to N_{\text{ORD}}$ is definable in $V$, because  it is definable from $\vec{U}$ and we are assuming that $\vec{U}$ is definable in $V$.  For every $i \in \text{ORD}$, $N_i$  is wellfounded by Theorem 19.30 of \cite{MR1940513}.  It follows that the membership relation of $N_{\text{ORD}}$ is also wellfounded, and as usual we may without loss of generality assume that
\[
N_{\text{ORD}}^{\text{trans}}:=\big\{  y \in N_{\text{ORD}} \ : \ \text{The } \in_{N_{\text{ORD}}} \text{-transitive closure of } y \text{ is a set} \big\}
\]
is transitive.  The construction of the iteration ensures that
\begin{equation}
\forall i < j \text{ both in } \text{ORD,} \ \text{crit}\big( \pi_{i,j} \big) = \pi_{0,i}(\kappa_i)
\end{equation}
and it follows that
\begin{equation}\label{eq_CritPt}
\forall i \in \text{ORD} \ \text{crit}(\pi_{i,\text{ORD}})= \pi_{0,i}(\kappa_i) 
\end{equation}

To prove that $N_{\text{ORD}}^{\text{trans}}$ is in fact all of $N_{\text{ORD}}$, it suffices to check that an  arbitrary $N_{\text{ORD}}$-ordinal $y$ is an element of $\text{ORD}$.   Pick $i \in \text{ORD}$ such that $y \in \text{rng}(\pi_{i, \text{ORD}})$, and let $y_i:= \pi_{i,\text{ORD}}^{-1}(y)$.  Pick a limit $\ell \in \text{ORD}$ such that $\ell \ge i$ and $\pi_{0,i}(\kappa_\ell) > y_i$.  Then applying $\pi_{i,\ell}$ to both sides yields that $\pi_{0,\ell}(\kappa_\ell) > \pi_{i,\ell}(y_i) = \pi_{\ell, \text{ORD}}^{-1}(y)$.  Since the critical point of $\pi_{\ell,\text{ORD}}$ is $\pi_{0,\ell}(\kappa_\ell)$ by \eqref{eq_CritPt}, it follows that $y =  \pi_{\ell, \text{ORD}}^{-1}(y) \in \text{ORD}^{N_\ell}= \text{ORD}$.

 Finally note that for each $i \in \text{ORD}$, $\pi_{0,\text{ORD}}$ is discontinuous at  $\kappa_i$ because 
\[
\pi_{0,\text{ORD}} = \pi_{i,\text{ORD}} \circ \pi_{0,i}
\]
and $\text{crit}(\pi_{i,\text{ORD}})= \pi_{0,i}(\kappa_i)$ by \eqref{eq_CritPt}.  Since each $\kappa_i$ is measurable and hence regular, it follows that $\pi_{0,\text{ORD}}$ is not eventually continuous.

So $\pi_{0,\text{ORD}}$ and the trivial forcing witness that P generically holds.  On the other hand, suppose for a contradiction that P ideally holds.  Then there is some normal system of filters $\vec{F}$ and some generic $G$ for $\mathbb{B}_{\vec{F}}$ such that $V[G]$ believes that the generic ultrapower $j:V \to N$ has property P.  By Theorem 2.37 of Audrito-Steila~\cite{AudritoSteila}, the generic ultrapower $j: V \to N$ can be viewed as an extender embedding in $V[G]$.  That is, for a sufficiently large $\lambda$ we have 
\[
N = \{ j(f)(s) \ : \ s \in [\lambda]^{<\omega} \text{ and } f \in V \cap {}^{[\lambda]^{|s|}} V \}.
\]
Then the following standard argument shows that $j$ is continuous on $\text{cof}^V(> \lambda)$, contradicting that $P(j)$ holds.  Assume $\text{cf}^V(\theta) > \lambda$ and let $j(f)(s) < j(\theta)$, where $f \in V \cap {}^{[\lambda]^{|s|}} V$ and $s \in [\lambda]^{<\omega}$.  Without loss of generality we can assume that $\text{rng}(f) \subseteq \theta$.  Since $\text{cf}^V(\theta) > \lambda$ then $\text{sup}\big( \text{rng}(f)  \big) < \theta$ and hence $j\Big(\text{sup}\big( \text{rng}(f)  \big) \Big) < j(\theta)$.  Then
\[
j(f)(s) \le \text{sup}\big( \text{rng}\big( j(f) \big) \big) = j \big( \text{sup}(\text{rng}(f))  \big) <  j(\theta)
\]
So there is an element of $\text{range}(j)$ in the interval $[j(f)(s), j(\theta) \big)$, which completes the proof.
\end{proof}

\section{A characterization of supercompactness in terms of short extenders}\label{sec_CharSuper}

In this section we provide a characterization of supercompactness in terms of short extenders (Theorem \ref{thm_SC_Charact}).  This characterization allows us to show that NS is not necessarily universal for the class of all wellfounded extenders (Corollary \ref{cor_StrongNotSuper}).  One interesting feature of Theorem \ref{thm_SC_Charact} is that the characterization makes no reference to the strength of the extenders.\footnote{Recall the \textbf{strength} of an extender $\vec{F}$ is defined to be the largest ordinal $\eta$ such that $V_\eta \subset \text{ult}(V,\vec{F})$.  }

We will use the following characterization of supercompactness due to Magidor:
\begin{theorem}[Magidor~\cite{MR0327518}; see Viale~\cite{Viale_GuessingModel} for the formulation presented here]\label{thm_Magidor}
The following are equivalent for a regular uncountable cardinal $\kappa$:
\begin{enumerate}
 \item $\kappa$ is supercompact.
 \item For class-many $\lambda > \kappa$ there are stationarily many $M \in \wp_\kappa(H_\lambda)$ such that $M \cap \kappa \in \kappa$ and $M$ collapses to a hereditary initial segment of the universe.\footnote{i.e.\ $\text{otp}(M \cap \text{ORD})$ is a cardinal and the transitive collapse of $M$ is $H_{\text{otp}(M \cap \text{ORD})}$.}
\end{enumerate}
\end{theorem}

\begin{theorem}\label{thm_SC_Charact}
The following are equivalent:
\begin{enumerate}
 \item\label{item_supercompact} $\kappa$ is supercompact;
 \item\label{item_ShortExtendersProjections} For class-many $\lambda$ there exists a short $(\kappa,\lambda)$-extender that is a canonical projection of NS.
\end{enumerate}
\end{theorem}
\begin{proof}
Assume first that $\kappa$ is supercompact.  Given $\lambda > \kappa$, let $U$ be a normal fine measure on $\wp_\kappa(\lambda)$, and let $j: V \to_U N$ be the ultrapower embedding.  Let $\vec{F}$ be the $(\kappa,\lambda)$-extender derived from $j$; i.e.\ for each $a \in [\lambda]^{<\omega}$, 
\[
X \in F_a \iff a \in j(X)
\]
Then $U \ge \vec{F}$ by Lemma \ref{lem_Char_I_below}, so by   \ref{lem_Equiv_ProjNS} $\vec{F}$ is a canonical projection of $U$.  Since $j(\kappa) > \lambda$ then $\vec{F}$ is a short extender, i.e.\ each $F_a$ concentrates on $[\kappa]^{|a|}$.  By Burke's Theorem \ref{thm_Burke}, $U$ is a canonical projection of $\text{NS} \restriction S_{U \text{-good}}$.  Clearly the relation ``X is the canonical projection of Y" (as in Definition \ref{def_PO}) is transitive, so $\vec{F}$ is the canonical projection of $\text{NS} \restriction S_{U \text{-good}}$.

Now assume \ref{item_ShortExtendersProjections}.  By Magidor's Theorem \ref{thm_Magidor}, it suffices to prove that for all regular $\theta > \kappa$, there are stationarily many $M \in \wp_\kappa(H_\theta)$ such that $M \cap \kappa \in \kappa$, $\text{otp}(M \cap \text{ORD})$ is a cardinal, and $M$ collapses to a hereditary initial segment of the universe.  Fix a regular $\theta > \kappa$ and an algebra $\mathfrak{A}$ on $H_\theta$.  By \ref{item_ShortExtendersProjections} there is a $\lambda > 2^\theta$ and a short $(\kappa,\lambda)$-extender $\vec{F}$ that is a canonical projection of NS.  By Lemma \ref{lem_Equiv_ProjNS}, the set $S_{\vec{F}\text{-good}}$ of $\vec{F}$-good structures is stationary in $H_{\lambda^+}$.  Fix any $\vec{F}$-good $M$ such that 
\[
M \prec (H_{\lambda^+}, \in, \mathfrak{A}, \vec{F})
\]

Let $\sigma_M: H_M \to M \prec H_{\lambda^+}$ be the inverse of the Mostowski collapse of $M$ and $\theta_M:= \sigma_M^{-1}(\theta)$.  Note that $M_\theta:= M \cap H_\theta \prec \mathfrak{A}$ and that $\sigma_M^{-1}(H_\theta) = \sigma_M^{-1}[M_\theta] = \big( H_{\theta_M} \big)^{H_M} $ is the transitive collapse of $M_\theta$.  Now by Fact \ref{fact_TransBelowKappa}, $M \cap \kappa \in \kappa$.  
We will prove that
\begin{equation}\label{eq_ToProveContained}
\wp^{H_M}(\theta_M) = \wp^V(\theta_M)
\end{equation}
which will imply that $\text{otp}(M_\theta \cap \text{ORD}) = \text{ht}\big( H_{M_\theta} \big)$ is a cardinal in $V$ and that $(H_{\theta_M})^{H_M} = (H_{\theta_M})^V$.  Since $M$ is $\vec{F}$-good then by Observation \ref{obs_DerivedIsElement} the $(\sigma_M^{-1}(\kappa), \sigma_M^{-1}(\lambda))$-extender $\vec{F}^M$ derived from $\sigma_M$ is an element of $H_M$.  Let $H'_M:= \text{ult}(H_M, \vec{F}^M)$ and $k_M: H'_M \to H_{\lambda^+}$ the canonical factor map, so that
\[
\sigma_M = k_M \circ j_{\vec{F}^M}
\]
Note that since $\vec{F}^M \in H_M$ then the ultrapower map $j_{\vec{F}^M}$ is internal to $H_M$, so in particular
\begin{equation}\label{eq_UltIsInternal}
H_M \supset H'_M
\end{equation}

Since $\vec{F}^M$ has length $\lambda_M:= \sigma_M^{-1}(\lambda)$ then $\text{crit}(k_M) \ge \lambda_M$.  Also, because $\lambda > 2^\theta$ then by elementarity of $\sigma_M$, $(2^{\theta_M})^{H_M} < \lambda_M$, and by \eqref{eq_UltIsInternal} it follows that $(2^{\theta_M})^{H'_M} < \lambda_M$.  Putting all this together, we have
\[
(2^{\theta_M})^{H'_M} < \text{crit}(k_M)
\]
which implies that
\begin{equation}\label{eq_k_preservesPowerSet}
\wp^{H'_M}(\theta_M) = \wp^{H_{\lambda^+}}(\theta_M) = \wp^V(\theta_M)
\end{equation}
Putting \eqref{eq_k_preservesPowerSet} and \eqref{eq_UltIsInternal} together yields \eqref{eq_ToProveContained} and completes the proof.
\end{proof}

\begin{corollary}\label{cor_StrongNotSuper}
Suppose $\kappa$ is strong but not supercompact.  Then for class-many $\lambda$ there is a short extender with critical point $\kappa$ and strength $\lambda$ that is not a canonical projection of NS.
\end{corollary}
\begin{proof}
Since $\kappa$ is not supercompact, then by Magidor's Theorem \ref{thm_Magidor} there is some regular $\theta > \kappa$ such that there are only nonstationarily many $M \in \wp_\kappa(V_\theta)$ such that $M \cap \kappa \in \kappa$ and $M$ collapses to a hereditary initial segment of the universe; let $A_{\kappa,\theta}$ denote this nonstationary set.  Let $\lambda > 2^\theta$.  Since $\kappa$ is $\lambda$-strong, there exists a short extender $\vec{F}$ with critical point $\kappa$ and strength $\lambda$; and without loss of generality the length of $\vec{F}$ is at least $\lambda$.  By the proof of the \ref{item_ShortExtendersProjections} $\implies$ \ref{item_supercompact} direction of Theorem \ref{thm_SC_Charact}, if $\vec{F}$ were a canonical projection of NS then $A_{\kappa,\theta}$ would be stationary, which is a contradiction.   
\end{proof}

\section{I3 towers and `` \texorpdfstring{$\aleph_\omega$}{alephOmega} is Jonsson"}\label{sec_OmegaCofTowers}

If $\lambda$ is a singular cardinal of cofinality $\omega$, $\vec{I} = \langle I_x \ : \ x \in V_\lambda \rangle$ is a tower of ideals with domain $V_\lambda$, and $\langle \kappa_n \ : \ n \in \omega \rangle$ is increasing and cofinal in $\lambda$, then $\vec{I}$ is generated by
\[
\langle I_{V_{\kappa_n}} \ : \ n \in \omega \rangle
\]
because every $x \in V_\lambda$ is a subset (in fact member) of some $V_{\kappa_{n(x)}}$, and $I_x$ is the canonical projection of $I_{V_{\kappa_{n(x)}}}$ to an ideal on $x$.\footnote{Moreover if $\lambda$ is strong limit then one could just use $I_{\kappa_n}$ instead of $I_{V_{\kappa_n}}$; i.e.\ in that scenario $\langle I_x \ : \ x \in V_\lambda \rangle$ can be recovered from just $\langle I_{\kappa_n} \ : \ n \in \omega \rangle$.}  For this reason, to simplify notation we will often say things like ``suppose $\langle I_n \ : \ n \in \omega \rangle$ is a tower of height $\lambda$", where it is understood that $I_n= I_{V_{\kappa_n}}$ for some $\vec{\kappa} = \langle \kappa_n \ : \ n \in \omega \rangle$ that is cofinal in $\lambda$; this does not depend on the choice of the cofinal sequence (since the sequences indexed by any two cofinal sequences in $\lambda$ generate the same tower).

\begin{lemma}\label{lem_TreeGoodStruc}
Suppose $\langle \kappa_n \ : \ n \in \omega \rangle$ is a strictly increasing sequence of cardinals with supremum $\lambda$, and $\vec{I}=\langle I_n \ : \ n \in \omega \rangle$ is a tower of normal ideals of height $\lambda$ (where $I_n$ is an ideal on $V_{\kappa_n}$).  Suppose $W$ is a set such that $V_\lambda \subset W$ and $|W|=|V_\lambda|$, $F: [W]^{<\omega} \to W$ a function, $m \in \omega$, and $T \in I_m^+$.  Fix any surjection $\phi: V_\lambda \to W$ with the property that
\[
\forall n \in \omega \ \ V_{\kappa_n} \subseteq \phi[V_{\kappa_n}]
\]
Then there is a tree
\[
\mathcal{S}^{\vec{I},F,m,T,\phi}
\]  
of height $\le \omega$ such that the following are equivalent:
\begin{enumerate}
 \item There exists some $M \subseteq W$ such that $M$ is $\vec{I}$-good, $M$ is closed under $F$, and $M \cap V_{\kappa_m} \in T$.
 \item There is a cofinal branch through $\mathcal{S}^{\vec{I},F,m,T,\phi}$. 
\end{enumerate}
Moreover this equivalence is absolute to any outer model.
\end{lemma}
\begin{proof}
For each $n \in \omega$ let $W_n:= \phi[V_{\kappa_n}]$.  

Let $\mathcal{S}^{\vec{I},F,m,T,\phi}$ be the collection of all finite sequences
\[
X_0, X_1, \dots, X_k
\]
such that:
\begin{enumerate}
 \item $X_i \subset W_i$ for all $i \le k$;
 \item If $m \le k$ then $X_m \cap V_{\kappa_m} \in T$;
 \item For all $i \le j \le k$:     \begin{enumerate}
  \item $X_i = X_j \cap W_i$;
  \item $X_i \cap V_{\kappa_i} = X_j \cap V_{\kappa_i}$;
 \item For all $z \in [X_i]^{<\omega}$: if $F(z) \in W_j$ then $F(z) \in X_j$;
 \item For all $C \in W_j \cap \text{Dual}(I_i)$:  $X_i \cap V_{\kappa_i} \in C$; 
\end{enumerate}
\end{enumerate}
Clearly $\mathcal{S}^{\vec{I},F,m,T,\phi}$ is closed under initial segments, and is hence a tree.  If $\langle X_n \ : \ n \in \omega \rangle$ is a cofinal branch, then $X_\omega:= \bigcup_{n < \omega} X_n$ is closed under $F$, the intersection of $X_\omega$ with $V_{\kappa_m}$ is in $T$, and $X_\omega$ is $\vec{I}$-good.

Conversely, if $Y \subset W$ and $Y \cap V_{\kappa_m} \in T$, $Y$ is closed under $F$, and $Y$ is $\vec{I}$-good, then 
\[
\langle  Y \cap W_n \ : \ n \in \omega \rangle
\]
is a cofinal branch through the tree.  

The ``moreover" part follows from absoluteness of wellfoundedness between transitive models of set theory, together with the fact that: for a given $M$ the statement ``$M$ is $\vec{I}$-good, closed under $F$, and $M \cap V_{\kappa_m}=\text{sprt}(I_m) \in T$" is a $\Sigma_0$ statement in the parameters $M$, $\vec{I}$, $F$, $V_{\kappa_m}$, and $T$.   
\end{proof}

\begin{theorem}\label{thm_Equiv_WF_project}
Let $\vec{I}=\langle I_n \ : \ n \in \omega \rangle$ be an $\omega$-cofinal tower of normal ideals.  
\begin{enumerate}
 \item\label{item_PrecipImpliesProjection} If $\vec{I}$ is precipitous, then it is a canonical projection of NS. 
 \item\label{item_WFultraEquiv} If the dual of each $I_n$ is an ultrafilter, then ``$\vec{I}$ has wellfounded ultrapower" is equivalent to ``$\vec{I}$ is a canonical projection of NS".
\end{enumerate}
\end{theorem}
\begin{proof}
First fix a cofinal sequence $\vec{\kappa} = \langle \kappa_n \ :  \ n \in \omega \rangle$ of cardinals in $\lambda$ (so that, by the convention mentioned earlier, $I_n = I_{V_{\kappa_n}}$).

To prove part \ref{item_PrecipImpliesProjection}, suppose $\vec{I}$ is precipitous.  To prove that $\vec{I}$ is a canonical projection of NS, by Lemma \ref{lem_Equiv_ProjNS} it suffices to prove that there are $\vec{I}$-projective-stationarily many $\vec{I}$-good substructures of $H_\theta$ (for some sufficiently large $\theta$); i.e.\ for every $T \in \vec{I}^+$ there are stationarily many $M \prec H_\theta$ such that $M \cap \cup T \in T$ and
\[
\forall n \in \omega \ \forall C \in M \cap \text{Dual}(I_n) \  M \cap \text{sprt}(I_n) \in C
\]
Let $n(T)$ be such that $\cup T = V_{\kappa_{n(T)}}$.  Let $\mathfrak{A} = (H_\theta, \in, \dots)$ be an algebra.  Let $\lambda$ be the height of $\vec{I}$ and fix some $W \prec \mathfrak{A}$ such that $V_\lambda \subset W$ and $|W|=|V_\lambda|$.  Fix a surjection $\phi: V_\lambda \to W$ as in the assumptions of Lemma \ref{lem_TreeGoodStruc}.  Set $\mathfrak{A}_W:= \mathfrak{A} \restriction W$ and for each $n$ let $W_n:= \phi[V_{\kappa_n}]$.  Fix a function $F:[W]^{<\omega} \to W$ such that any subset of $W$ closed under $F$ is elementary in $\mathfrak{A}_W$.  It suffices to find some $\vec{I}$-good subset of $W$ that is closed under $F$ and whose intersection with $V_{\kappa_{n(T)}} = \text{sprt}(I_{n(T)})$ is in $T$.

Let $G$ be $(V,\mathbb{B}_{\vec{I}})$-generic with $T \in G$ and $j: V \to N$ the generic ultrapower.  Since $\vec{I}$ is precipitous then $N$ is wellfounded.    Then $j[V_{\kappa_{n(T)}}] \in j(T)$ because $T \in G$.  In $N$ consider the tree $\mathcal{S}^{j(\vec{I}),j(F),n(T),j(T),j(\phi)}$ from Lemma \ref{lem_TreeGoodStruc}.  From the point of view of $V[G]$, every proper initial segment of the sequence
\[
b:= \langle  j\big[ W_n \big] \ : \ n \in \omega \rangle
\]
is an element of $N$ (because $j[\text{sprt}(I_n)] = j[V_{\kappa_n}] \in N$ for every $n \in \omega$) and moreover every proper initial segment of $b$ is an element of the tree $\mathcal{S}^{j(\vec{I}),j(F),n_T,j(T),j(\phi)}$.  Since $N$ is wellfounded in $V[G]$ and $V[G]$ believes that the tree $\mathcal{S}^{j(\vec{I}),j(F),n_T,j(T),j(\phi)}$ has a cofinal branch, then by absoluteness there exists a cofinal branch in $N$.  By Lemma \ref{lem_TreeGoodStruc}, $N$ believes there is a $j(\vec{I})$-good set closed under $j(F)$ whose intersection with $j(\kappa_{n(T)})$ is in $T$.  By elementarity of $j$ we are done with the proof of part \ref{item_PrecipImpliesProjection}.

For part \ref{item_WFultraEquiv}, assume $\vec{U} = \langle U_n \ : \ n \in \omega \rangle$ is an $\omega$-cofinal tower of normal ultrafilters.  If $\vec{U}$ has wellfounded ultrapower then it is a canonical projection of NS, by part \ref{item_PrecipImpliesProjection}.  Now suppose $\vec{U}$ is a canonical projection of NS.  By Lemma \ref{lem_Equiv_ProjNS} there are stationarily many $\vec{U}$-good structures in $H_\theta$ (for $\theta$ sufficiently large).  Fix some $\vec{U}$-good $M \prec (H_\theta, \in, \vec{U})$ and let $\sigma_M: H_M \to H_\theta$ be the inverse of the Mostowski collapsing map of $M$.  By Observation \ref{obs_DerivedIsElement}, the tower of ultrafilters $\vec{U}^M$ derived from $\sigma_M$ is the same as the preimage $\vec{U}_M:= \sigma_M^{-1}(\vec{U})$ of $\vec{U}$ in $H_M$.  Since $\text{ult}(H_M, \vec{U}^M)$ can always be embedded into $H_\theta$ and is hence wellfounded, then $\text{ult}(H_M, \vec{U}_M)$ is a wellfounded ultrapower that is internal to $H_M$.  So $H_M \models$ ``$\vec{U}_M$ is wellfounded", so by elementarity of $\sigma_M$, $H_\theta \models$ ``$ \sigma_M(\vec{U}_M) = \vec{U}$ is wellfounded".
\end{proof}

Recall that an \textbf{I3 embedding} is a nontrivial elementary embedding from $V_\lambda \to V_\lambda$ for some $\lambda$.  An \textbf{I2 embedding} is a nontrivial elementary embedding $j: V \to N$ where $N$ is transitive and $V_{\kappa_\omega} \subset N$, where $\kappa_\omega$ is the supremum of the critical sequence.\footnote{Recall the critical sequences is defined inductively by setting $\kappa_0:= \text{crit}(j)$ and $\kappa_{n+1}:= j(\kappa_n)$ for all $n \in \omega$.}

\begin{definition}\label{def_I2_tower}
We say that $\vec{I}$ is an \textbf{I3 tower} iff it is an $\omega$-cofinal tower of normal ideals and $\mathbb{B}_{\vec{I}}$ forces that the critical sequence of the generic ultrapower is contained in, and cofinal in, $\text{ht}(\vec{I})$.\footnote{I.e.\ it is forced that $\text{crit}(j_{\dot{G}}) < \text{ht}(\vec{I})$, $\dot{\kappa}_n:= j^n(\text{crit}(j_{\dot{G}}) )< \text{ht}(\vec{I})$ for every $n \in \omega$, and $\langle \dot{\kappa}_n \ : \ n \in \omega \rangle$ is cofinal in $\text{ht}(\vec{I})$.}

An \textbf{I3 tower of ultrafilters} is an I3 tower such that the dual of every ideal in the tower is an ultrafilter.
\end{definition}


It is straightforward to see that for a regular uncountable $\kappa$, the following are equivalent:
\begin{enumerate}
 \item There exists an I3 embedding with critical point $\kappa$;
 \item There exists an I3 tower of ultrafilters with critical point $\kappa$.
\end{enumerate}
The proof essentially appears in Kanamori~\cite{MR1994835}, but we provide a very brief sketch.  If $j: V_\theta \to V_\theta$ has critical point $\kappa$, let $\kappa_\omega$ be the supremum of the critical sequence $\langle \kappa_n \ : \ n \in \omega \rangle$.  Then $j \restriction V_{\kappa_\omega}$ is an elementary embedding from $V_{\kappa_\omega} \to V_{\kappa_\omega}$.  For each $n$ let 
\[
U_n:= \{ A \subseteq V_{\kappa_n} \ :  \  j[V_{\kappa_n}] \in j(A)  \}
\]  
Then $\vec{U}$ is an I3 tower of ultrafilters of height $\kappa_\omega$.  Conversely, suppose $\vec{U} = \langle U_n \ : \ n \in \omega \rangle$ is an I3 tower of ultrafilters of height $\lambda$ with critical point $\kappa$.  Let $j: V \to_{\vec{U}} N$ be the ultrapower of $V$ by $\vec{U}$; then $N$ is wellfounded below $\lambda$, and by definition (noting that $\mathbb{B}_{\vec{U}}$ is the trivial forcing), the critical sequence is cofinal in $\lambda$.  It follows that $j \restriction V_\lambda$ is an elementary embedding from $V_\lambda \to V_\lambda$.

On the other hand, I3 towers which happen to be canonical projections of NS yield I2 embeddings, and vice-versa:

\begin{theorem}\label{thm_CharI2_Project}
Let $\kappa$ be regular uncountable.  The following are equivalent:
\begin{enumerate}
 \item\label{item_IsCritI2} $\kappa$ is the critical point of some I2 embedding.
 \item\label{item_ExistsI2towerWF} There exists an I3 tower of ultrafilters of completeness $\kappa$ with wellfounded ultrapower.
 \item\label{item_ExistsI2towerProj} There exists an I3 tower of ultrafilters of completeness $\kappa$ that is a canonical projection of NS.
\end{enumerate}
\end{theorem}
\begin{proof}
The equivalence of \ref{item_IsCritI2} with \ref{item_ExistsI2towerWF} is standard, see Chapter 24 of Kanamori~\cite{MR1994835}.  The equivalence of \ref{item_ExistsI2towerWF} with \ref{item_ExistsI2towerProj} follows from part \ref{item_WFultraEquiv} of Theorem \ref{thm_Equiv_WF_project}.
\end{proof}

Although uncountably cofinal towers of ultrafilters always yield wellfounded ultrapowers, this is not the case with $\omega$-cofinal towers of ultrafilters:
\begin{lemma}\label{lem_IllfoundedI2}
If there is an I2 embedding, then there is an I3 tower of ultrafilters with illfounded ultrapower.
\end{lemma}
\begin{proof}
Let $\kappa$ be the least cardinal such that there exists a wellfounded I2 embedding with critical point $\kappa$.  Without loss of generality we may assume this embedding is an ultrapower by an I3 tower of ultrafilters $\vec{U} = \langle U_n \ : \ n \in \omega \rangle$.  Let $\lambda$ be the height of $\vec{U}$, and $j: V \to_{\vec{U}} N$ the ultrapower map, where $N$ is transitive.  Then $V_\lambda \subset N$.  In $N$ consider the tree $\mathcal{S}$ of all finite, partial I3 towers of ultrafilters in $V_\lambda$ with critical point $\kappa$, ordered by extension.  Every proper initial segment of $\vec{U}$ is an element of $N$ and hence an element of $\mathcal{S}$.  Hence $\vec{U}$ is a cofinal branch through $\mathcal{S}$, and by absoluteness $N$ believes there is a cofinal branch through $\mathcal{S}$.  So 
\[
N \models \ \text{There exists an I3 tower of ultrafilters with critical point } \kappa
\]
Since $\kappa < j(\kappa)$ then $V$ believes there exists an I3 tower of ultrafilters with critical point $<\kappa$.  By minimality of $\kappa$, this I3 tower must be illfounded.
\end{proof}

Notice that if $\vec{I}$ is an I3 tower of height $\aleph_\omega$, then the requirement in Definition \ref{def_I2_tower} that the critical sequence is forced to be cofinal in $\aleph_\omega^V$ is redundant;  it follows already from the requirement that the critical sequence is forced to be contained in $\aleph_\omega^V$.  To see this, suppose $G$ is $\mathbb{B}_{\vec{I}}$-generic over $V$, $j_G: V \to N_G$ is the generic ultrapower, $\langle \kappa_n \ : \ n \in \omega \rangle$ is the critical sequence, and $\{ \kappa_n \ : \ n \in \omega  \} \subset \aleph^V_\omega$.  Then $\kappa_0 = \text{crit}(j_G)$ is a cardinal in $V$, and $H^V_{\aleph_\omega} \subset N_G$ by basic properties of tower ultrapowers.  Using the assumption that $\{ \kappa_n \ : \ n \in \omega \} \subset \aleph^V_\omega$, an inductive proof then yields that $\kappa_{n+1} = j_G(\kappa_n)$ is a cardinal in both $V$ and $N_G$ for every $n \in \omega$.  It follows that $\aleph_\omega^V$ is the supremum of $\vec{\kappa}$.

\begin{definition}
Given a normal system of filters $\vec{F}$, we say that $\boldsymbol{\vec{F}}$ \textbf{decides its critical sequence} iff there is a sequence $\vec{\kappa} = \langle \kappa_n \ : \ n \in \omega \rangle$ (in the ground model) such that $\mathbb{B}_{\vec{F}}$ forces $\check{\vec{\kappa}}$ to be the critical sequence of the generic ultrapower map.  

\end{definition}

The following lemma provides a characterization of I3 towers that decide their critical sequence without referring to the forcing relation.

\begin{lemma}\label{lem_DecidesCritImplies}
Suppose $\lambda$ is a singular cardinal of cofinality $\omega$ and $\vec{I}=\langle I_x \ : \ x \in V_\lambda \rangle$ is a tower of ideals of height $\lambda$.  Suppose $\vec{\kappa} = \langle \kappa_n \ : \ n \in \omega \rangle$ is increasing and cofinal in $\lambda$.  The following are equivalent:
 \begin{enumerate}
  \item\label{eq_IsI2TowerAndDecides} $\vec{I}$ is an I3 tower, and decides its critical sequences as $\vec{\kappa}$
  \item\label{eq_ConcentrateC_n} $I_{\kappa_0}$ is $\kappa_0$-complete, and for every $n \ge 1$:
\[
C_n:=\{ X \subset \kappa_n \ : \ \forall i \in [1,n] \ \text{otp}(X \cap \kappa_i) = \kappa_{i-1} \} \in \text{Dual}(I_{\kappa_n}) 
\] 
 \end{enumerate}
\end{lemma}
\begin{proof}
For the \eqref{eq_IsI2TowerAndDecides} $\implies$ \eqref{eq_ConcentrateC_n} direction:  that $I_{\kappa_0}$ is $\kappa_0$-complete follows easily from the assumption that $\kappa_0$ is forced to be the critical point of generic ultrapowers by $\vec{I}$.  Now fix $n \ge 1$.  Let $G$ be an arbitrary generic for $\mathbb{B}_{\vec{I}}$ and $j: V \to N$ the generic ultrapower; then $\vec{\kappa}$ is the critical sequence of $j$ by assumption.  Then
\[
\forall i \in [1,n] \ \ j[\kappa_n] \cap j(\kappa_i) = j[\kappa_i]  \text{ and } \text{otp}\big( j[\kappa_i] \big) = \kappa_{i-1}
\]
so $j[\kappa_n] \in j(C_n)$.  Since $G$ was arbitrary, it follows that $C_n \in \text{Dual}(I_{\kappa_n})$.

For the \eqref{eq_ConcentrateC_n} $\implies$ \eqref{eq_IsI2TowerAndDecides}  direction, assume $G$ is $\mathbb{B}_{\vec{I}}$-generic over $V$ and $j: V \to N$ is the generic ultrapower.  We just need to prove that $\vec{\kappa}$ is the critical sequence of $j$; the background assumption that $\vec{\kappa}$ is cofinal in $\lambda$ takes care of the remaining requirement of Definition \ref{def_I2_tower}.  First, the $\kappa_0$-completeness of $I_{\kappa_0}$ ensures that $\kappa_0 \le \text{crit}(j)$.  Now $C_1 \in \text{Dual}(I_{\kappa_1})$, so $C_1 \in G$; hence  $j[\kappa_1] \in j(C_1)$ so $j[\kappa_1]$ has ordertype exactly $\kappa_0$, which implies the other inequality $\kappa_0 \ge \text{crit}(j)$.  Hence
\begin{equation}\label{eq_Kappa0}
\text{crit}(j) = \kappa_0 \text{ and } j(\kappa_0) = \kappa_1
\end{equation}

Let $\langle \kappa'_n \ : \ n \in \omega \rangle$ enumerate the sequence $j(\vec{\kappa})$; so $\kappa'_\ell = \kappa_{\ell+1}$ for every $\ell \in \omega$.  Now for any $n \ge 1$, since $C_n \in G$ then $j[\kappa_n] \in j(C_n)$, so by the definition of $j(C_n)$ we have
\[
\forall n \ge 1 \ \ \forall i \in [1,n] \ \ \text{otp}\big( j[\kappa_n] \cap \kappa'_i \big) =  \kappa'_{i-1}  
\]    
and it follows that $j(\kappa'_{i-1}) = \kappa'_i$ for all $i \ge 1$; equivalently that $j(\kappa_\ell) = \kappa_{\ell+1}$ for every $\ell \ge 1$.  Together with \eqref{eq_Kappa0} this yields that $\vec{\kappa}$ is the critical sequence of $j$.
\end{proof}

Silver proved the following theorem, which we phrase in modern terminology:
\begin{theorem}[Silver; see Foreman-Magidor~\cite{MR1846032}]\label{thm_Silver}
Assume $2^\omega < \aleph_\omega$.  The following are equivalent (only the \ref{item_Silver_Jonsson} $\implies$ \ref{item_Silver_MS} direction uses the assumption that $2^\omega < \aleph_\omega$):
\begin{enumerate}
 \item\label{item_Silver_Jonsson} $\aleph_\omega$ is Jonsson
 \item\label{item_Silver_MS} There exists some pair of sequences $\langle \kappa_i \ : \ i \in \omega \rangle$ and $\langle \mu_i \ : \ i \in \omega \rangle$, both subsequences of the $\aleph_n$'s, such that 
 \[
 \langle \kappa_i \cap \text{cof}(\mu_i) \ : \ i \in \omega \rangle
 \]
 is mutually stationary.\footnote{Mutual stationarity was introduced in Foreman-Magidor~\cite{MR1846032}.  In the current setting, mutual stationarity of $\langle \kappa_i \cap \text{cof}(\mu_i) \ : \ i \in \omega \rangle$ means that for every $F: [\aleph_\omega]^{<\omega} \to \aleph_\omega$ there exists some set $X$ closed under $F$ such that $\text{cf}\big(\text{sup}(X \cap \kappa_i)\big)=\mu_i$ for all $i \in \omega$. }
\end{enumerate}
\end{theorem}

We provide another characterization of ``$\aleph_\omega$ is Jonsson" (assuming $2^\omega < \aleph_\omega$), from which Silver's theorem can be easily derived:
\begin{theorem}\label{thm_ImproveSilver}
Assume $2^\omega < \aleph_\omega$.  The following are equivalent (the assumption $2^\omega < \aleph_\omega$ is only used for the \ref{item_Sean_Jonsson} $\implies$ \ref{item_Sean_ProjectionNS} direction):
\begin{enumerate}
 \item\label{item_Sean_Jonsson} $\aleph_\omega$ is Jonsson;
 \item\label{item_Sean_ProjectionNS} There exists an I3 tower of height $\aleph_\omega$ that decides its critical sequence, \textbf{and} is a canonical projection of NS.
\end{enumerate}
\end{theorem}

We will need a standard fact:
\begin{fact}
Suppose $\theta \ge \aleph_\omega$, $\mathfrak{A}$ is a Skolemized structure extending $(H_\theta,\in)$ in a countable language, and $M \prec \mathfrak{A}$.  Let $\mu$ be a cardinal below $ \aleph_\omega$ and 
\[
M':= \text{Sk}^{\mathfrak{A}}(M \cup \mu)
\]
Then for all $n$ such that $\omega_n > \mu$:  
\[
\text{sup}(M' \cap \omega_n) = \text{sup}(M \cap \omega_n)
\]
\end{fact}
\begin{proof}
Let $n$ be such that $\omega_n > \mu$ and suppose $\alpha \in M' \cap \omega_n$.  Then $\alpha = f(\xi)$ for some function $f \in M$ and some ordinal $\xi < \mu$.  Then without loss of generality, $f: \mu \to \omega_n$.  Since $\mu < \omega_n$ and $f \in M$ then
\[
\text{sup}(\text{rng}(f)) \in M \cap \omega_n
\]
and hence $\alpha < \text{sup}(M \cap \omega_n)$. 
\end{proof}

\begin{corollary}\label{cor_JonssonWithReals}
If $2^\omega < \aleph_\omega$ and $\aleph_\omega$ is Jonsson, then for all $\theta \ge \aleph_\omega$ there are stationarily many $M \prec H_\theta$ such that $M \cap \aleph_\omega \subsetneq \aleph_\omega$, $|M \cap \aleph_\omega| = \aleph_\omega$, and $M$ contains all the reals.
\end{corollary}

\begin{proof}
(of Theorem \ref{thm_ImproveSilver}):  First we prove the easier direction.  Assume $\vec{I}$ is a tower as in part \ref{item_Sean_ProjectionNS}, and that $\langle \kappa_n \ : \ n \in \omega \rangle$ is the critical sequence it decides; note that $\vec{\kappa}$ must be a subsequence of the $\aleph_n$'s by the remarks above.  Let $\theta$ be large and $\mathfrak{A}$ an algebra on $H_\theta$.  We need to find some $M \prec \mathfrak{A}$ such that $|M \cap \aleph_\omega|=\aleph_\omega$ but $\aleph_\omega \nsubseteq M$.  Without loss of generality, $\mathfrak{A}$ includes a predicate for the sequence $\langle C_n \ : \ n \in \omega \rangle$ where $C_n$ is as defined in Lemma \ref{lem_DecidesCritImplies}.  Since we assume $\vec{I}$ is a canonical projection of NS, by Lemma \ref{lem_Equiv_ProjNS} there is some $M \prec \mathfrak{A}$ that is $\vec{I}$-good.  Since $C_n \in M$ for all $n$ and $M$ is $\vec{I}$-good, then $M \cap \kappa_n \in C_n$ for all $n \in \omega$.  This implies that
\[
\forall i \ge 1 \ \ \text{otp}(M \cap \kappa_i) = \kappa_{i-1} 
\]
so in particular
\[
\forall i \ge 1 \ \ \text{cf}(\text{sup}(M \cap \kappa_i)) = \kappa_{i-1}
\]
This implies that $|M \cap \aleph_\omega| \ge \text{sup}_i \kappa_{i-1} = \aleph_\omega$, yet clearly there are gaps in $M$ below $\aleph_\omega$.  In fact we've just proved that the sequence 
\[
\langle \kappa_i \cap \text{cof}(\kappa_{i-1}) \ : \ i \ge 1 \rangle
\]
is mutually stationary (i.e.\ part  \ref{item_Silver_MS} of Theorem \ref{thm_Silver}).

For the direction \ref{item_Sean_Jonsson} $\implies$ \ref{item_Sean_ProjectionNS} we assume $2^\omega < \aleph_\omega$ and that $\aleph_\omega$ is Jonsson.  Fix a large regular $\theta$, and let $S$ denote the set of $\aleph_\omega$-Jonsson subsets of $H_\theta$ that include all the reals; $S$ is stationary by Corollary \ref{cor_JonssonWithReals}.  For each $M \in S$ let $\sigma_M: H_M \to H_\theta$ be the inverse of the collapsing map of $M$.  Then $\aleph_\omega$ is a fixed point of $\sigma_M$, and hence the critical sequence $\langle \kappa^M_n \ : \ n \in \omega \rangle$ of $\sigma_M$ is a subset of $\aleph_\omega$.  Also $\sigma_M(\kappa^M_{n-1}) = \kappa^M_{n}$ is a cardinal in $V$ for every $n \ge 1$.  For each $n \ge 1$ let $k^M_n$ be the natural number such that $\kappa^M_n = \omega_{k^M_n}$.  Since $M$ includes all the reals, then
\[
s_M:= \langle k^M_n \ : \ n \ge 1 \rangle \in M
\]
and hence
\[
t_M:= \langle \kappa^M_n \ : \ n \ge 1 \rangle = \langle \omega_{s_M(n)} \ : \ n \ge 1 \rangle  \in M
\]

So the map $M \mapsto t_M$ is a regressive function on $S$.  By Fodor's Lemma there is a fixed $t$ that is a subsequence of the $\omega_n$'s and a stationary $T \subseteq S$ such that $t_M = t$ for all $M \in T$.  Notice also that since $t \in M$ for $M \in T$, it makes sense to apply $\sigma^{-1}$ to $t$, and it easily follows that
\[
\forall M \in T \ \ \langle \kappa^M_n \ : \ n \in \omega \rangle  \in H_M
\]

Let $\vec{I}$ be the projection of $\text{NS} \restriction T$ to a tower of height $\aleph_\omega$; more precisely, for $n \ge 1$ let $I_n$ be the projection of $\text{NS} \restriction T$ to an ideal on $V_{\kappa_n}$, where $\langle \kappa_n \ : \ n \ge 1  \rangle$ is the increasing enumeration of $t$ (starting at 1).  Note that
\begin{equation}\label{eq_MatchKappas}
\forall M \in T  \ \ \forall n \ge 1 \ \ \kappa^M_n = \kappa_n
\end{equation}
Clearly $\vec{I}$ is a tower, and is a canonical projection of NS by definition.  It remains to prove that it is an I3 tower and that it decides its critical sequence.  By Lemma \ref{lem_DecidesCritImplies} it suffices to prove that $I_1$ is $\kappa_1$-complete and that
\[
\forall n \ge 2 \ \ \{  X \subset \kappa_n \ : \  \forall i \in [2,n] \ \  \text{otp}(X \cap \kappa_i) = \kappa_{i-1}  \} \in \text{Dual}(I_n)
\]  
which in turn (by definition of ideal projection) is equivalent to proving that the following holds for all $n \ge 2$:
\[
\{ M \in T \ : \  M \cap \kappa_1 \in \kappa_1 \text{ and } \forall i \in [2,n] \ \  \text{otp}(M \cap \kappa_{i}) = \kappa_{i-1}    \} \in \text{Dual}\big( \text{NS} \restriction T \big)
\]

But in fact for every $M \in T$:
\begin{itemize}
 \item $M \cap \kappa_1 \in \kappa_1$ because $\kappa_1 = \kappa^M_1 =  \sigma_M(\text{crit}(\sigma_M))$; and
 \item if $i \ge 2$ then  
 \[
 M \cap \kappa_i = M \cap \kappa^M_i =  \sigma_M[\kappa^M_{i-1}]
 \]
which clearly has ordertype $\kappa^M_{i-1}$; and since $i \ge 2$ then \eqref{eq_MatchKappas} ensures that $\kappa^M_{i-1} = \kappa_{i-1}$.
\end{itemize} 

\end{proof}

Theorem \ref{thm_ImproveSilver}, in conjunction with Woodin's result on the stationary tower, gives the following corollary, which is similar to Corollary 3.2 of Burke~\cite{MR1472122}:
\begin{corollary}
Suppose $2^\omega < \aleph_\omega$, $\aleph_\omega$ is Jonsson, and there exists a Woodin cardinal.  Then in some forcing extension of $V$ there is an elementary embedding $j:V \to N$ with $N$ wellfounded, $\text{crit}(j) < \aleph_\omega^V$, $\aleph_\omega^V$ is a fixed point of $j$, and the critical sequence of $j$ is an element of $V$.
\end{corollary}
\begin{proof}
By Theorem \ref{thm_ImproveSilver}, there exists a tower $\vec{I}$ of height $\aleph_\omega$ that decides its own critical sequence and is a canonical projection of NS.  Since $\vec{I}$ is a canonical projection of NS, by Lemma \ref{lem_Equiv_ProjNS} $S^\theta_{\vec{I}\text{-good}}$ is stationary, where $\theta:= 2^{\aleph_\omega}$.  Let $\delta$ be a Woodin cardinal, and let $G$ be generic for the full stationary tower $\mathbb{P}_{\delta}$ with $S^\theta_{\vec{I} \text{-good}} \in G$.  Let $j:V \to_G N$ be the generic ultrapower; by Larson~\cite{MR2069032}, $N$ is wellfounded and
\[
j[H_\theta^V] \in N
\]

In $V[G]$ let $\vec{\kappa} = \langle \kappa_n \ :  \ n \in \omega \rangle$ be the critical sequence of $j$.\footnote{We do not know yet that this is the same as the critical sequence decided by $\vec{I}$, because $G$ is generic for the stationary tower, \textbf{not} for $\mathbb{B}_{\vec{I}}$.  However they do end up being the same critical sequence.}  Now every $M \in S^\theta_{\vec{I}\text{-good}}$ has the property that $\sigma_M: H_M \to H_\theta$ has fixed point $\aleph_\omega^{H_M}$, $\text{crit}(\sigma_M) <\aleph^{H_M}_\omega$, and the critical sequence of $\sigma_M$ is an element of $H_M$.  By Los' Theorem, $N$ believes these same statements about $M:= j[H_\theta^V]$.  Since $H_{M}= H^V_\theta$ and $\sigma_{M} = j \restriction H^V_\theta$ it follows that $\vec{\kappa} \in H^V_\theta$.
\end{proof}

%

\section{Questions}\label{sec_Questions}

Several results in this paper showed that, in the presence of large cardinals, NS can consistently fail to be universal for certain normal systems of filters.  For example, if there is a strong cardinal that is not supercompact, then many of its extenders are not canonical projections of NS (Corollary \ref{cor_StrongNotSuper}); and if there is an I2 embedding then there exists an I3 tower of ultrafilters that is not a canonical projection of NS (conjunction of Theorem \ref{thm_CharI2_Project} with Lemma \ref{lem_IllfoundedI2}). 

On the other hand, Burke's ZFC theorem shows that all normal ideals and normal towers of inaccessible height (precipitous or otherwise) are canonical projections of NS.  Moreover the only normal systems of filters that provably exist in ZFC that the author is aware of are those that are explicitly defined as projections of normal ideals,\footnote{E.g.\ one can fix a stationary set $S$ and canonically project $\text{NS} \restriction S$ to another ideal, a tower of ideals, or an ideal extender.} which are hence (by Burke's theorem) canonical projections of NS.  This raises the following question:
\begin{question}\label{q_ZFC}
Is it consistent with ZFC that every normal system of filters is a canonical projection of NS?  
\end{question}

The referee suggested another related question, motivated by the results in Section \ref{sec_CharSuper}:
\begin{question}\label{q_Referee}
Is the existence of a strong, non-supercompact cardinal an anti large cardinal axiom?  I.e.\ is it the case that assuming as many large cardinal axioms as one wishes in the universe, any strong cardinal must be supercompact?
\end{question} 

(\textbf{Added in press: } Stamatis Dimopoulos pointed out to the author that ``there exists a strong cardinal, and every strong cardinal is supercompact" is not consistent, because by the proof of Proposition 26.11 of Kanamori~\cite{MR1994835}, if $\kappa$ is supercompact then there exists a $\mu < \kappa$ and some superstrong embedding $\pi$ with $\text{crit}(\pi) = \mu$ and $\pi(\mu) = \kappa$.  By composing with embeddings witnessing supercompactness of $\kappa$, it follows that $\mu$ is a strong cardinal.  In particular, if $\kappa$ is the least supercompact cardinal then there is always a strong, non-supercompact cardinal below it.)

In fact, the author does not even know if the statement $\Phi \equiv$ ``there exists a strong cardinal, and every strong cardinal is supercompact" is even consistent.  By  Corollary \ref{cor_StrongNotSuper}, if $\kappa$ is strong but not supercompact, then for all sufficiently large $\lambda$, all (short) $\lambda$-strong extenders with critical point $\kappa$ must fail to be canonical projections of NS.  In particular, any model with at least one strong cardinal that witnesses a positive solution to Question \ref{q_ZFC} would also witness the consistency of $\Phi$.

We turn now to questions about the material in Section \ref{sec_OmegaCofTowers}.  It is straightforward to show that if $\vec{F}$ is a normal system of filters and the set of $\vec{F}$-self-generic structures is $\vec{F}$-projective stationary, then $\vec{F}$ is precipitous.  Theorem 3.8 of \cite{Cox_MALP} gave a converse to this, in the case where $\vec{F}$ is a single ideal whose universe consists of countable sets.\footnote{E.g. a normal fine ideal $I$ on $[H_\theta]^\omega$ is precipitous if and only if there are $I$-projective-stationarily many $I$-self-generic structures.}  Does this converse extend to $\omega$-cofinal towers?   Theorem \ref{thm_Equiv_WF_project} shows that if $\vec{I}$ is a precipitous, $\omega$-cofinal tower of normal ideals, then there are $\vec{I}$-projective-stationarily many $\vec{I}$-good structures.\footnote{And in fact this characterizes wellfoundedness if each $I_n$ is the dual of an ultrafilter.}  The following question asks if this can be improved to get self-generic structures:
\begin{question}
Suppose $\vec{I}$ is a precipitous, $\omega$-cofinal tower of ideals.  Must the set of $\vec{I}$-self-generic structures be $\vec{I}$-projective stationary?
\end{question}

Recall from the discussion after Definition \ref{def_I2_tower} that I3 embeddings were characterized by I3 towers of ultrafilters, whereas I2 embeddings were characterized by I3 towers of ultrafilters that happen to be canonical projections of NS.  Also, Theorem \ref{thm_ImproveSilver} showed (under assumption of small continuum) that ``$\aleph_\omega$ is Jonsson" is equivalent to the existence of an I3 tower of height $\aleph_\omega$ that decides its critical sequence and is a canonical projection of NS.  This analogy suggests the following conjecture:

\begin{conjecture}
The phrase ``and is a canonical projection of NS" cannot be removed from item \ref{item_Sean_ProjectionNS} in the statement of Theorem \ref{thm_ImproveSilver}.  I.e.\ the existence of an I3 tower of height $\aleph_\omega$ that decides its critical sequence does \textbf{not} imply that  $\aleph_\omega$ is Jonsson.
\end{conjecture}

\end{document}